  \documentclass[11pt,letterpaper,twoside]{article}

  \usepackage[margin=3.05cm]{geometry}
  \usepackage{amsmath,amssymb,amsthm,amsfonts,bbding,stmaryrd,dsfont,wasysym,pifont,xspace,xcolor}
  \usepackage{parskip,fancyhdr,url}
  \usepackage{fourier}
  \usepackage{enumitem}
  \usepackage{mathrsfs}
  \usepackage[
     breaklinks,colorlinks,
     citecolor=blue,linkcolor=blue,urlcolor=teal
  ]{hyperref}
  \usepackage{tikz}
  \usetikzlibrary{
    arrows,
    calc
    }

  \begingroup
    \makeatletter
    \@for\theoremstyle:=definition,remark,plain\do{%
      \expandafter\g@addto@macro\csname th@\theoremstyle\endcsname{%
        \addtolength\thm@preskip\parskip
        }%
      }
  \endgroup

  \makeatletter
    \def\tagform@#1{\maketag@@@{%
     \textbf{(\ignorespaces#1\unskip\@@italiccorr)}}}%
     \renewcommand{\eqref}[1]{\textup{\maketag@@@{(\ignorespaces%
          {\ref{#1}}\unskip\@@italiccorr)}}}
  \makeatother

  \newcommand\address[1]{}
  \newcommand\email[1]{}
  \newcommand\dedicatory[1]{}

  \theoremstyle{plain}
  \newtheorem{theorem}[equation]{Theorem}
  \newtheorem{proposition}[equation]{Proposition}
  \newtheorem{corollary}[equation]{Corollary}
  \newtheorem{lemma}[equation]{Lemma}

  \theoremstyle{definition}
  
  \newtheorem{remark}[equation]{Remark}

  \numberwithin{equation}{section}

  \newcommand{\set}[1]{\ensuremath{\left\{ {#1} \right\}}\xspace} 
  \newcommand{\abs}[1]{\ensuremath{\left\lvert {#1} \right\rvert}\xspace} 
  \newcommand{\norm}[1]{\ensuremath{\left\lVert {#1} \right\rVert}\xspace} 
  \newcommand{\st}{\ensuremath{\,\, \colon \,\,}\xspace} 
  \newcommand{\from}{\ensuremath{\colon \thinspace}\xspace} 

  \DeclareMathOperator{\area}{Area}
  \DeclareMathOperator{\height}{height}
  \DeclareMathOperator{\width}{width}
  \DeclareMathOperator{\id}{id}
  \DeclareMathOperator{\polyspec}{Poly}
  \DeclareMathOperator{\hol}{hol}

  \newcommand{\Z}{\ensuremath{\mathbb{Z}}\xspace}
  \newcommand{\R}{\ensuremath{\mathbb{R}}\xspace}
  \newcommand{\CC}{\ensuremath{\mathbb{C}}\xspace}
  \newcommand{\SL}{\ensuremath{\mathrm{SL}(2,\R)}\xspace}
  \newcommand{\RP}{\R \mathrm{P}^1}

  \newcommand{\scc}{\ensuremath{\mathcal{S}}\xspace} 
  \newcommand{\T}{\ensuremath{\mathcal{T}(S)}\xspace} 
   
  \newcommand{\mcg}{\ensuremath{\text{MCG}}\xspace}  
  \newcommand{\MF}{\ensuremath{\mathcal{MF}}\xspace} 
  \newcommand{\PMF}{\ensuremath{\mathcal{PMF}}\xspace} 
  \newcommand{\QD}{\ensuremath{\text{QD}}\xspace} 
  \newcommand{\Qk}{\ensuremath{\text{Q}(\kappa)}\xspace} 
  \newcommand{\QMod}{\ensuremath{\text{MQD}}\xspace} 
  \newcommand{\QkMod}{\ensuremath{\text{MQ}(\kappa)}\xspace} 
  \newcommand{\orb}{\ensuremath{\mathcal{O}}\xspace}
  \newcommand{\poly}{\ensuremath{\mathrm{P}}\xspace}
  \newcommand{\aim}{\ensuremath{\mathrm{M}}\xspace}
  \newcommand{\aimup}{\ensuremath{\widetilde{\mathrm{M}}}\xspace}
  \newcommand{\cyl}{\ensuremath{\mathsf{cyl}}\xspace}
  \newcommand{\hcyl}{\ensuremath{\widehat{\mathsf{cyl}}}\xspace}
  \newcommand{\convexspace}{\ensuremath{\mathscr{K}_0}\xspace}
  \newcommand{\cont}{\ensuremath{\mathscr{C}}\xspace}
  \newcommand{\lsl}{\ensuremath{l^{\,\mathrm{SL}}\xspace}}
  \newcommand{\VT}{\ensuremath{\text{VT}}\xspace} 
  \newcommand{\PA}{\ensuremath{\text{PA}}\xspace} 
  \newcommand{\IL}{\ensuremath{\text{IL}}\xspace} 
  \newcommand{\sv}{\ensuremath{\lambda}\xspace}
  \newcommand{\ecc}{\ensuremath{\mathrm{ecc}}\xspace} 

  \newcommand{\param}{{\mathchoice{\mkern1mu\mbox{\raise2.2pt\hbox{$
  \centerdot$}}
  \mkern1mu}{\mkern1mu\mbox{\raise2.2pt\hbox{$\centerdot$}}\mkern1mu}{
  \mkern1.5mu\centerdot\mkern1.5mu}{\mkern1.5mu\centerdot\mkern1.5mu}}}

  \fancypagestyle{plain}
  
  \AtEndDocument{\bigskip{\footnotesize%
  Mathematics Department, University of Oklahoma, Norman, OK 73019, USA
  \par
  \texttt{forester@math.ou.edu}, \texttt{rtang@math.ou.edu},
  \texttt{jing@math.ou.edu} 
  }}

\begin{document}

  \title    {Veech surfaces and simple closed curves} \author   {Max
  Forester, Robert Tang, and Jing Tao} \date{}

  \maketitle \thispagestyle{empty}

  \begin{abstract} 

    We study the $\SL$--infimal lengths of simple closed curves on
    half-translation surfaces. Our main result is a characterization of
    Veech surfaces in terms of these lengths. 

    We also revisit the ``no small virtual triangles'' theorem of Smillie
    and Weiss and establish the following dichotomy: the virtual triangle
    area spectrum of a half-translation surface either has a gap above zero
    or is dense in a neighborhood of zero. 
    
    These results make use of the \emph{auxiliary polygon} associated to a
    curve on a half-translation surface, as introduced by Tang and Webb.
    
  \end{abstract}
  
\section{Introduction}
  
  Let $S$ be a surface of genus $g$ with $p$ marked points and let $\QD(S)$
  be the space of quadratic differentials on $S$. Each element $q \in
  \QD(S)$ naturally endows $S$ with a locally Euclidean metric with
  isolated conical singularities and linear holonomy restricted to
  $\set{\pm \id}$. We also refer to elements of $\QD(S)$ as
  half-translation surfaces. There is a natural action of $\SL$ on $\QD(S)$
  preserving the signature of the singularities. A half-translation surface
  $q$ is called a \emph{Veech surface} if its group of (derivatives of)
  affine self-diffeomorphisms is a lattice in $\SL$. Veech surfaces possess
  remarkable dynamical properties akin to flat tori, and arise naturally in
  the contexts of rational billiards and Teichm\"uller curves in moduli
  space \cite{Vee89}.
  
  In this paper, we study $\QD(S)$ from the point of view of simple closed
  curves on $S$. On a half-translation surface $q$, a simple closed curve
  $\alpha$ either has a unique geodesic representative or there is a
  maximal flat cylinder on $q$ foliated by closed geodesics in the homotopy
  class of $\alpha$. In the former case, the geodesic representative
  $\alpha^q$ of $\alpha$ is a concatenation of saddle connections. We say
  that $\alpha$ is a \emph{crooked curve} on $q$ if $\alpha^q$ has at least
  two saddle connections whose associated holonomy vectors are not
  parallel. We define the $\SL$--\emph{infimal length} of $\alpha$ on $q$
  to be
  \[
    \lsl_\alpha(q) = \inf_{q' \in \, \SL \, \cdot \, q} l_{\alpha}(q'),
  \]
  where $l_\alpha(q)$ denotes the geodesic length of $\alpha$ on $q$. A
  curve $\alpha$ is crooked on $q$ if and only if $\lsl_\alpha(q)$ is
  positive (see Proposition \ref{polygon}). Our first main result is a
  characterization of Veech surfaces in terms of their $\SL$--infimal
  length spectra.
  
  \begin{theorem}\label{nscc}

   Let $q$ be a half-translation surface. Then $q$ is a Veech surface if
   and only if it has no short crooked curves: there is an $\epsilon >0$
   such that $\lsl_\alpha(q) \geq \epsilon$ for every crooked curve
   $\alpha$.

  \end{theorem}
  
  This result is reminiscent of the no-small-(virtual)-triangles theorem
  due to Smillie and Weiss \cite{SW10}. They characterize Veech surfaces as
  the half-translation surfaces which possess a positive lower bound on the
  areas of Euclidean triangles on $q$ with edges formed by saddle
  connections. One advantage of working with simple closed curves is that
  they are topological objects; they do not depend on the half-translation
  structure, and therefore we can study them at once over the entire
  quadratic differential space. In contrast, a saddle connection on $q$
  persists only in an open subset of the relevant stratum
  of $\QD(S)$. 
  
  One of the main tools we use in this paper is the \emph{auxiliary
  polygon} $P_\alpha(q)$ associated to a simple closed curve $\alpha$ on
  $q$, as introduced by the second author and Webb in \cite{TW15}. The area
  of $P_\alpha(q)$ gives an estimate for $\lsl_\alpha(q)^2$ up to bounded
  multiplicative error (see Proposition \ref{polygon}\ref{tw3}). It follows
  that Theorem \ref{nscc} is equivalent to the statement that $q$ is a
  Veech surface if and only if the \emph{polygonal area spectrum} 
  \[
    \polyspec(q) = \set{\area(P_\alpha(q)) \st \alpha \text{ is a simple
    closed curve}} 
  \] 
  has a gap above zero. In fact, our arguments will prove a slightly
  stronger statement: either there is a gap (exactly when $q$ is a Veech
  surface), or this spectrum is dense in a neighborhood $[0,a)$ for some $a
  > 0$. This statement also holds for the \SL--infimal length spectrum.
  
  The forward implication of Theorem \ref{nscc} will follow from a
  relatively straightforward application of the no-small-virtual-triangles
  theorem.  
  
  The reverse implication of Theorem \ref{nscc} will make use of the
  auxiliary polygon mentioned above, as well as two other ingredients: the
  orbit closure theorem of Eskin, Mirzakhani, and Mohammadi \cite{EMM15}
  and a rigidity statement for $\SL$--orbits of quadratic differentials due
  to Duchin, Leininger, and Rafi \cite{DLR10}. We remark that the orbit
  closure theorem is used only to deduce local path connectedness of
  $\SL$--orbit closures in strata of half-translation surfaces. 

  To prove the reverse implication we use an elementary continuity
  argument; see Proposition \ref{notveech-dense}. An essential step is 
  establishing that the auxiliary polygon $P_\alpha(q)$ is continuous in
  $q$, with respect to the Hausdorff topology. This is achieved in several
  steps. By continuity of the intersection pairing between measured
  foliations on $S$, we obtain a continuous function {$\QD^1(S) \to
  \cont(\RP, \R)$ of the form $q \mapsto i\big(\nu_q^{\frac{\pi}{2}
  +\theta}, \alpha\big)$ for any fixed curve $\alpha$ on $S$. The value
  $i\big(\nu_q^{\frac{\pi}{2} +\theta}, \alpha\big)$ coincides with the
  \emph{width} of the auxiliary polygon in direction $\theta$, and so the
  \emph{width function} $w_{P_\alpha(q)}\in \cont(\RP,\R)$ of the polygon
  is continuous in $q$. Finally, standard results from convex geometry on
  centrally symmetric sets yield continuity of $P_\alpha(q)$ itself. 

  As a further application of continuity of the auxiliary polygon, we
  show that $\SL$--infimal length is continuous in $q$; see Proposition
  \ref{infimalcontinuity}.  

  Next consider the \emph{virtual triangle area spectrum}, defined as
  follows: 
  \[
    \VT(q) = \set{ \, \abs{u \wedge v} \st u, v \in \hol(q)} 
  \]
  where $\hol(q)$ is the set of holonomy vectors of saddle
  connections in $q$. The no-small-virtual-triangles theorem of
  \cite{SW10} states that $\VT(q)$ has a gap above zero if and only if
  $q$ is a Veech surface. In our second theorem we show further that
  $\VT(q)$ resembles the polygonal area spectrum as discussed above: 

  \begin{theorem} \label{nsvt}

    Let $q$ be a half-translation surface. Then $\VT(q)$ either has a
    gap above zero (exactly when $q$ is a Veech surface) or is
    dense in a neighborhood $[0,a)$ for some $a>0$. 

  \end{theorem}

  In fact, our argument provides a new proof of the ``gap
  implies Veech'' direction of the no-small-virtual-triangles
  theorem. Moreover, the virtual triangles yielding the
  dense subset of $[0,a)$ in the non-Veech case are \emph{based}
  virtual triangles; see Section \ref{sec:triangles} and Proposition
  \ref{densetriangles}. 
  The non-Veech case of Theorem \ref{nsvt} is proved very similarly to
  Theorem \ref{nscc}. 

  Finally we have some additional results on polygonal area and
  $\SL$--infimal length when $q$
  is a Veech surface. The first of these will be derived from the
  analogous result for $\VT(q)$ in \cite{SW10}. 

  \begin{theorem} \label{discreteness}

    If $q$ is a Veech surface then $\polyspec(q)$ is a discrete subset
    of\/ $\R$.  

  \end{theorem}

  Next, for $a > 0$ define the sets 
  \begin{eqnarray*}
    \PA(a) &=& \set{q \in \QD^1(S) \st \area(\poly_\alpha(q)) \geq a \text{
    for every crooked curve } \alpha \text{ on } q}, \\
    \IL(a) &=& \set{q \in \QD^1(S) \st \lsl_\alpha(q) \geq a \text{
    for every crooked curve } \alpha \text{ on } q}.
  \end{eqnarray*}
  We say that two half-translation surfaces are \emph{affinely
  equivalent} if they are related by the actions of the mapping class
  group and $\SL$. 

  \begin{theorem} \label{finiteness}

    For any $a > 0$, the sets $\PA(a)$ and $\IL(a)$ both contain only finitely
    many affine equivalence classes of half-translation surfaces. 

  \end{theorem}

  This result is an application of Theorem 2.2 of \cite{EMM15}, which
  classifies the closed $\SL$--invariant subsets of strata of
  half-translation surfaces: namely, any such subset is a finite union of
  orbit closures. 
  
  \subsection*{Acknowledgements} The authors thank Alex Wright for many
  helpful conversations regarding $\SL$--orbit closures. We also thank
  Jenya Sapir for encouraging us to prove continuity of $\SL$--infimal
  length, and the referee for providing helpful feedback. Forester was
  partially supported by NSF award DMS-1105765, 
  and Tao by NSF award DMS-1311834. 
  
\section{Background}
  
  \subsection{Quadratic differentials and half-translation surfaces}
  
  We begin by recalling relevant background regarding quadratic
  differentials and half-translation surfaces. For further details, consult
  \cite{Str84}. 

  Let $S$ be a surface of genus $g$ and let $\varrho \subset S$ be $p$
  marked points with $3g-3+p \ge 1$. A \emph{half-translation
  structure} on $S$ consists of a finite set $\varsigma$ of singular
  points on $S$ (possibly including marked points),
  together with an atlas of charts to $\CC$ defined away from $\varsigma$,
  where the transition maps are of the form $z \mapsto \pm z + c$ for some
  $c \in \CC$. By pulling back the standard Euclidean metric on $\CC$, one
  obtains a locally Euclidean metric on $S - \varsigma$. We require that
  the metric completion of this metric yields $S$ with a singular Euclidean
  structure, where every singularity has a cone angle $n \pi$ for some $n
  \ge 1$, and those that are not marked points must have cone angle at
  least $3\pi$. The atlas determines a preferred vertical direction on $S$,
  and we consider this to be part of the data in the half-translation
  structure.

  One can construct a half-translation surface by taking a finite
  collection of disjoint Euclidean polygons in $\CC$, with pairs of edges
  identified by gluing maps of the form $z \mapsto \pm z + c$ \cite{FM14}.

  By a \emph{quadratic differential} $q$ on $S$, we mean a complex
  structure on $S$ equipped with an integrable holomorphic quadratic
  differential on $S - \varrho$ with finitely many zeros. This means that
  $q$ can extended to a meromorphic quadratic differential on $S$ with at
  worst simple poles at the marked points. The number of zeros minus the
  number of poles of $q$, counted with multiplicity, is $4g-4$. The union
  of the set of zeros and poles will constitute the singularities of $q$.
  There is a \emph{natural} holomorphic coordinate system $z$ on $S$ such
  that, in a neighborhood of a point away from a singularity, $q$ is given
  by $q = dz^2$. In a neighborhood of a zero of order $k \ge 1$, $q =
  z^kdz^2$, and in a neighborhood of a pole, $q = \frac{1}{z}dz^2$. A zero
  of order $k$ has cone angle $(k+2)\pi$ and a pole has cone angle $\pi$.
  Thus $q$ induces a half-translation structure on $S$, where the singular
  Euclidean metric is given locally by $|dz|^2$.  We shall use $q$ to
  denote a quadratic differential on $S$, as well as $S$ equipped with the
  corresponding half-translation structure. The assumptions on $q$ ensure
  that the area of the half-translation surface, equal to $\int_S
  \abs{q}dz^2$, is finite.

  Let $\QD(S)$ be the space of quadratic differentials on $S$. This is a
  complex manifold of dimension $6g-6+2p$, and can be identified with the
  cotangent bundle to Teichm\"uller space $\T$ via the natural
  projection $\QD(S) \rightarrow \T$ obtained by taking the underlying
  complex structure of each half-translation surface. The space
  $\QD^1(S)$ of \emph{unit area} quadratic differentials on $S$ can be
  identified with the unit cotangent bundle to $\T$.
 
  There is a natural $\SL$--action on $\QD(S)$ defined by post-composing
  the coordinate charts to $\CC$ by an $\R$--linear transformation. One can
  view this action by applying an element ${A \in \SL}$ to a defining set
  of polygons for a half-translation surface $q$ to obtain a new
  half-translation structure $A \cdot q$, noting that parallel edges of the
  same length remain so under $\SL$--deformations. Note that
  $\SL$--deformations preserve area, and so $\SL$ also acts naturally on
  $\QD^1(S)$.

  We will write $e^{i\theta}\cdot q$ as shorthand for
  $\begin{pmatrix}\cos\theta & -\sin\theta \\ \sin\theta & \cos\theta
  \end{pmatrix}\cdot q$; or in natural coordinates at a regular point,
  $e^{i\theta}\cdot q =
  e^{i\theta} dz^2$.
  
  \subsection{Geodesic representatives and measured foliations}
  
  Let $\scc$ be the set of (free homotopy classes of) essential simple
  closed curves on $S$. For any $\alpha \in \scc$, either $\alpha$ has a
  unique geodesic representative on $q$, or there is a unique maximal flat
  cylinder on $q$ foliated by the closed geodesics in the homotopy class of
  $\alpha$. In the former case, the geodesic representative of $\alpha$ is
  a concatenation of \emph{saddle connections} -- embedded geodesic arcs or
  loops with endpoints at singularities with no singularities on their
  interior. The angle between consecutive saddle connections is always
  at least $\pi$ on both sides. We shall use $\alpha^q$ to denote any
  geodesic representative of $\alpha$ on $q$.
  
  If $\alpha^q$ is a core curve of a flat cylinder, then we call $\alpha$ a
  \emph{cylinder curve} on $q$. Let $\cyl(q)$ denote the set of cylinder
  curves on $q$, and $\hcyl(q)$ the set of curves whose geodesic
  representatives have constant direction on $q$. Any curve in $\scc -
  \hcyl(q)$ is called a \emph{crooked curve} on $q$.
  
  Let $|dz|^2 = dx^2+dy^2$ be the singular Euclidean metric associated to
  $q$. We can consider several notions of length of a curve $\alpha$ on
  $q$. The Euclidean length $l_\alpha(q)$ of $\alpha^q$ is given by
  integrating $\alpha^q$ with respect to $|dz|$ (in local coordinates).
  Integrating $\alpha^q$ with respect to $\abs{dx}$ and $\abs{dy}$  (in
  local coordinates) gives the \emph{horizontal} and \emph{vertical}
  lengths $l_\alpha^H(q)$ and $l_\alpha^V(q)$ respectively. Finally, define
  the $\SL$--\emph{infimal length} of $\alpha$ with respect to $q$ to be
  \[
    \lsl_\alpha(q) = \inf_{q' \in \, \SL \, \cdot \, q} l_\alpha(q').
  \]
  This should be viewed as a measure of length of $\alpha$ with respect to
  the $\SL$--orbit of $q$, rather than with respect to $q$ itself.
      
  Using the natural coordinate of $q$, one can pull back the foliation of
  $\CC$ by lines in the direction $\theta \in \RP$ to obtain a measured
  foliation $\nu_q^\theta$ on $q$, where the transverse measure is the
  Euclidean distance between leaves. In particular, the horizontal and
  vertical directions respectively give rise to the \emph{horizontal} and
  \emph{vertical} foliations $\nu_q^H = \nu_q^0$ and $\nu_q^V =
  \nu_q^{\pi/2}$. The map $\QD(S) \times \RP \rightarrow \MF(S)$ defined by
  ${(q, \theta) \mapsto \nu_q^\theta}$ is continuous, where $\MF(S)$ is the
  space of measured foliations on $S$ \cite{HM79}. Let $\MF(q) = \set{t
  \cdot \nu_q^\theta \st \theta \in \RP, t \in \R_+}$. Let $\PMF(S)$ and
  $\PMF(q)$ be the projectivizations of $\MF(S)$ and $\MF(q)$,
  respectively. Note that these sets are invariant under
  $\SL$--deformations.
  
  The \emph{geometric intersection number} $i \from \scc \times \scc \to
  \R$ extends continuously to ${\MF(S) \times \MF(S)}$. For any curve
  $\alpha$, we have $i(\nu_q^H, \alpha) = l_\alpha^V(q)$ and $i(\nu_q^V,
  \alpha) = l_\alpha^H(q)$. See \cite{FLP79} for additional background on
  measured foliations. 
  
  \subsection{Auxiliary polygons}
  
  We now recall the construction from \cite{TW15} of the \emph{auxiliary
  polygon} $P_\alpha(q)$ associated to a curve $\alpha$ and a quadratic
  differential $q \in \QD(S)$.
  
  To each saddle connection $e$ on $q$, we assign a \emph{holonomy vector}
  ${v}_e \in \R^2$ which is parallel to it and is of the same length. This
  is well defined up to scaling by $\pm 1$. For consistency, we require
  that the direction $\theta(v_e)$ (that is, the oriented angle from the
  positive $x$--axis to $v_e$) lies in the interval $[0, \pi)$. Consider a
  geodesic representative $\alpha^q$ on $q$. If $\alpha$ is a cylinder
  curve, we may choose $\alpha^q$ to be a boundary component of the maximal
  flat cylinder with core curve $\alpha$. Let $m_e$ be the number of times
  $\alpha^q$ runs over the saddle connection $e$, in either direction.
  Define
  \[
    P_\alpha(q) = \set{\sum_e t_e m_e {v}_e \st -\frac{1}{2} \leq
    t_e \leq \frac{1}{2}} \subset \R^2, 
  \] 
  where the sum is taken over all saddle connections used by
  $\alpha^q$. 

  The set $P_\alpha(q)$ is a convex Euclidean polygon,
  unless $\alpha^q$ has constant direction, in which case $P_\alpha(q)$
  degenerates to a line segment parallel to $\alpha^q$ of length
  $l_\alpha(q)$. In particular, in the case of a cylinder curve, the
  definition of $P_\alpha (q)$ does not depend on the choice of
  boundary component. Moreover, this construction commutes with
  $\SL$--deformations: for all $A \in \SL$ we have $P_\alpha(A \cdot q) = A
  \cdot P_\alpha(q)$.
     
  \begin{proposition}[\cite{TW15}]\label{polygon}

    Given a curve $\alpha \in \scc$ and $q \in \QD^1(S)$, the auxiliary
    polygon $P_\alpha(q)$ defined above satisfies the following:
    \begin{enumerate}[label=\textup{(\roman*)}]

      \item The perimeter of $P_\alpha(q)$ is $2
        l_\alpha(q)$,\label{tw1} 

      \item $\height(P_\alpha(q)) = l_\alpha^V(q)$ and
      $\width(P_\alpha(q))
        = l_\alpha^H(q)$, \label{tw2} 

      \item $\pi \area(P_\alpha(q)) \leq \lsl_\alpha(q)^2 \leq
        8\area(P_\alpha(q))$, \label{tw3} 

      \item $\area(P_\alpha(q)) = 0$ if and only if $\alpha \in
        \hcyl(q)$. \label{tw4} \qed

    \end{enumerate}

  \end{proposition}
   
  Here, the \emph{perimeter} of a polygon $P$ is the length of its
  boundary $\partial P$. In the situation where $P$ degenerates to a line
  segment, we view $\partial P$ as a closed path traversing the line
  segment once in each direction.
  
  The next property of $P_\alpha(q)$ that we need is really 
  a statement about centrally symmetric convex polygons in
  general. Let $v_1, \dotsc, v_k$ be non-zero vectors in $\R^2$ whose 
  directions $\theta(v_i)$ satisfy 
  \[
    0 \leq \theta(v_1) \leq \dotsb \leq \theta(v_k) < \pi.
  \]
  Let $P_k$ be the convex set $\set{ \sum_{i=1}^k t_i v_i \st - 1/2
  \leq t_i \leq 1/2}$. 

  \begin{lemma}\label{tiling}

    The points $p_k = \sum_{i=1}^k v_i/2$  and $-p_k$ lie on the boundary
    $\partial P_k$. Furthermore, 

    \begin{enumerate}[label=\textup{(\roman*)}]

    \item \label{t1} the boundary arc $\partial^+P_k$, traveling
    counter-clockwise from $-p_k$ to $p_k$, is the path obtained by
    concatenating the vectors $v_1, \dotsc, v_k$, in order. Similarly,
    the boundary arc $\partial^- P_k$ from $p_k$ to $-p_k$ is the
    concatenation of $-v_1, \dotsc, -v_k$.  

    \item \label{t2} $P_k$ admits a tiling by the (possibly degenerate)
    parallelograms $v_i \times v_j$, with one copy of $v_i 
    \times v_j$ for each unordered pair $i \not= j$. 

    \end{enumerate}

  \end{lemma}

  \begin{proof}

    We proceed by induction on $k$, the case $k=1$ being trivial (with
    $P_1$ a line segment). 

    Recall that the \emph{Minkowski sum} of sets
    $A, B \subset \R^2$ is the set $A + B = \set{a + b \st a \in A, b \in
    B}$. Observe that $P_k$ is the Minkowski sum of
    $P_{k-1}$ with the line segment $S_k$ with endpoints $\pm
    v_k/2$. The description of $P_{k-1}$ given by the
    induction hypothesis allows us to see this sum clearly. The boundary
    of $P_k$ decomposes into four parts: the path $\partial^+ P_{k-1} +
    \set{ -v_k/2}$, the path $\partial^- P_{k-1} + \set{ v_k / 2}$,
    and two shifted copies of $S_k$; see Figure \ref{polyfig}. It is
    the assumption that $\theta(v_k)$ lies between $\theta(v_i)$ and
    $\pi$ for all $i$ that ensures that $\partial P_k$ is as
    described. 

    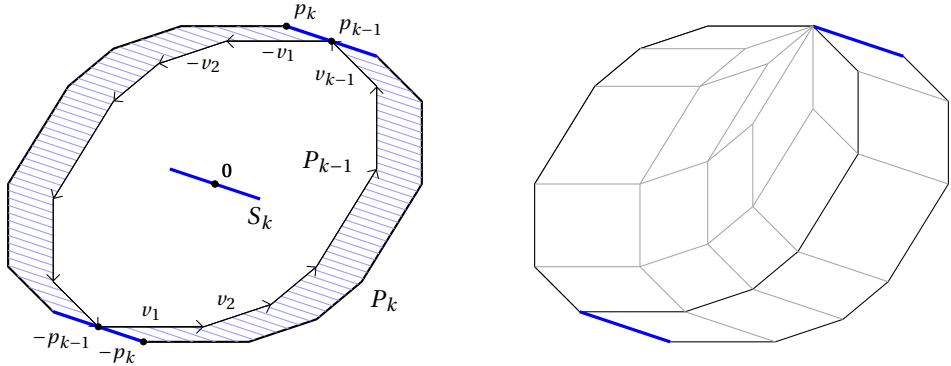
\begin{figure}[ht]
      \begin{center}
      \begin{tikzpicture}

        \path (0,0) coordinate (origin);
        \path (7,0) coordinate (rightorigin); 
        \path (1.4,0) coordinate (V1);
        \path (.9,.3) coordinate (V2);
        \path (.6,.5) coordinate (V3);
        \path (.8, 1.3) coordinate (V4);
        \path (0, 1.1) coordinate (V5);
        \path (-.6, .6) coordinate (V6);

        \path (-1.4,0) coordinate (V1-);
        \path (-.9,-.3) coordinate (V2-);
        \path (-.6,-.5) coordinate (V3-);
        \path (-.8, -1.3) coordinate (V4-);
        \path (0, -1.1) coordinate (V5-);
        \path (.6, -.6) coordinate (V6-);

        \path (.6, -.2) coordinate (W-);
        \path (-.6, .2) coordinate (W+);

        \path ($2*(W-)$) coordinate (V7-);
        \path ($2*(W+)$) coordinate (V7+);

        \path (origin) ++($.5*(V1-)$) ++($.5*(V2-)$) ++($.5*(V3-)$)
        ++($.5*(V4-)$) ++($.5*(V5-)$) ++($.5*(V6-)$) coordinate (P6-); 
        \path ($-1*(P6-)$) coordinate (P6+); 

        \draw[thick] (P6-) -- ++(W-) -- ++(V1) --
        ++(V2) -- ++(V3) -- ++(V4) -- ++(V5) -- ++(V6) -- ++(W+) -- ++(W+) --
        ++(V1-) -- ++(V2-) -- ++(V3-) -- ++(V4-) -- ++(V5-) -- ++(V6-) --
        ++(W-) -- cycle; 

        \begin{scope}
        \clip (P6-) -- ++(W-) -- ++(V1) --
        ++(V2) -- ++(V3) -- ++(V4) -- ++(V5) -- ++(V6) -- ++(W+) -- ++(W+) --
        ++(V1-) -- ++(V2-) -- ++(V3-) -- ++(V4-) -- ++(V5-) -- ++(V6-) --
        ++(W-) -- cycle; 
          \foreach \x in {-20,...,40}
          {\draw[blue!40] 
           (-3,.1*\x) -- (3,.1*\x-2);
          }
        \end{scope}

        \filldraw[fill=white, draw=black] (P6-) -- ++(V1) --
        ++(V2) -- ++(V3) -- ++(V4) -- ++(V5) -- ++(V6) -- 
        ++(V1-) -- ++(V2-) -- ++(V3-) -- ++(V4-) -- ++(V5-) -- ++(V6-) --
        cycle; 

        \small
        \draw (origin) ++(W-) node[anchor=north] {$S_k$};

        \scriptsize
        \draw[very thick,blue] (origin) ++(W-) -- ++(V7+);
        \filldraw (origin) circle (.4mm) node[anchor=south west] {$0$};

        \draw[->,>=angle 90,black] (P6-) -- ++(V1);
        \draw (P6-) -- ++(V1) +($-.5*(V1)$) node[anchor=south] {$v_1$};
        \draw[->,>=angle 90,black] (P6-) ++(V1) -- ++(V2); 
        \draw (P6-) ++(V1) ++(V2) +($-.5*(V2)$) +(-0.6,-.15) node[anchor=south] {$v_2$}; 
        \draw[->,>=angle 90,black] (P6-) ++(V1) ++(V2) -- ++(V3);
        \draw[->,>=angle 90,black] (P6-) ++(V1) ++(V2) ++(V3) -- ++(V4);
        \draw[->,>=angle 90,black] (P6-) ++(V1) ++(V2) ++(V3) ++(V4) --
        ++(V5);
        \draw[->,>=angle 90,black] (P6-) ++(V1) ++(V2) ++(V3) ++(V4) ++(V5) -- ++(V6);

        \draw[->,>=angle 90,black] (P6+) -- ++(V1-);
        \draw (P6+) -- ++(V1-) +($-.5*(V1-)$) node[anchor=north] {$-v_1$};
        \draw[->,>=angle 90,black] (P6+) ++(V1-) -- ++(V2-); 
        \draw (P6+) ++(V1-) ++(V2-) +($-.5*(V2-)$) +(0.6,.15)
        node[anchor=north] {$-v_2$};  
        \draw[->,>=angle 90,black] (P6+) ++(V1-) ++(V2-) -- ++(V3-);
        \draw[->,>=angle 90,black] (P6+) ++(V1-) ++(V2-) ++(V3-) -- ++(V4-);
        \draw[->,>=angle 90,black] (P6+) ++(V1-) ++(V2-) ++(V3-) ++(V4-) --
        ++(V5-);
        \draw[->,>=angle 90,black] (P6+) ++(V1-) ++(V2-) ++(V3-) ++(V4-)
        ++(V5-) -- ++(V6-); 

        \draw (P6+) +(0.05,-.3) node[anchor=north] {$v_{k-1}$}; 

        \filldraw (origin) circle (.4mm) node[anchor=south west] {$0$};

        \draw[very thick,blue] (P6-) ++(W-) -- ++(V7+) (P6+) ++(W-) -- ++(V7+);

        \filldraw (P6+) circle (.4mm) node[anchor=south west] {$p_{k-1}$};
        \filldraw (P6+) +(W+) circle (.4mm) node[anchor=south west] {$p_k$};

        \filldraw (P6-) circle (.4mm) node[anchor=north east] {$-p_{k-1}$};
        \filldraw (P6-) +(W-) circle (.4mm) node[anchor=north east] {$-p_k$};

        \draw (rightorigin) ++(P6-) -- ++(W-) -- ++(V1) --
        ++(V2) -- ++(V3) -- ++(V4) -- ++(V5) -- ++(V6) -- ++(W+) -- ++(W+) --
        ++(V1-) -- ++(V2-) -- ++(V3-) -- ++(V4-) -- ++(V5-) -- ++(V6-) --
        ++(W-) -- cycle; 

        \filldraw[fill=white, draw=black] (rightorigin) ++(P6-) ++(W+) -- ++(V1) --
        ++(V2) -- ++(V3) -- ++(V4) -- ++(V5) -- ++(V6) -- 
        ++(V1-) -- ++(V2-) -- ++(V3-) -- ++(V4-) -- ++(V5-) -- ++(V6-) --
        cycle; 

        \draw[black!35] (rightorigin) ++(P6-) ++(W+) -- +(V7-) ++(V1) -- +(V7-) 
        ++(V2) -- +(V7-) ++(V3) -- +(V7-) ++(V4) -- +(V7-) ++(V5) -- +(V7-) ;

        \draw[very thick,blue] (rightorigin) ++(P6-) ++(W-) -- ++(V7+)
        (rightorigin) ++(P6+) ++(W-) -- ++(V7+);

        \small
        \draw (origin) ++(1.5,.3) node {$P_{k-1}$}; 

        \draw (P6-) ++(W-) ++(V1) ++(V2) ++(V3) node[anchor=north west] {$P_k$}; 

        \draw[black!35] (rightorigin) ++(P6-) ++(W+) ++(V6) -- ++(V1) --
        ++(V2) -- ++(V3) -- ++(V4) -- ++(V5);

        \draw[black!35] (rightorigin) ++(P6-) ++(W+) ++(V6) ++(V1) -- +(V6-) 
        ++(V2) -- +(V6-) ++(V3) -- +(V6-) ++(V4) -- +(V6-) ;

        \draw[black!35] (rightorigin) ++(P6-) ++(W+) ++(V6) ++(V5) -- ++(V1) --
        ++(V2) -- ++(V3) -- ++(V4) ;

        \draw[black!35] (rightorigin) ++(P6-) ++(W+) ++(V6) ++(V5) ++(V1) -- +(V5-) 
        ++(V2) -- +(V5-) ++(V3) -- +(V5-) ;

        \draw[black!35] (rightorigin) ++(P6-) ++(W+) ++(V6) ++(V5) ++(V4) --
        ++(V1) -- ++(V2) -- ++(V3) ;

        \draw[black!35] (rightorigin) ++(P6-) ++(W+) ++(V6) ++(V5) ++(V4)
        ++(V1) -- +(V4-) ++(V2) -- +(V4-) ;

        \draw[black!35] (rightorigin) ++(P6-) ++(W+) ++(V6) ++(V5) ++(V4)
        ++(V3) -- ++(V1) -- ++(V2) ;

        \draw[black!35] (rightorigin) ++(P6-) ++(W+) ++(V6) ++(V5) ++(V4)
        ++(V3) ++(V1) -- +(V3-) ;

      \end{tikzpicture}
      \end{center}
      \caption{The polygon $P_k$ expressed as $P_{k-1} + S_k$,
        and tiled by parallelograms $v_i \times v_j$.} 
      \label{polyfig}

    \end{figure}

    The point $p_k$ is the initial point of the path $\partial^- P_{k-1}
    + \set{ v_k / 2}$, which is on the boundary. Similarly, $-p_k$ is the
    initial point of $\partial^+ P_{k-1} + \set{ -v_k/2}$, and conclusion
    \ref{t1} is now clear. 

    Next divide $P_k$ into two regions, one bounded by $\partial
    P_{k-1} + \set{v_k/2}$, and the other bounded by the paths
    $\partial^+ P_{k-1} + \set{v_k/2}$, $\partial^+ P_{k-1} +
    \set{-v_k/2}$, and the two copies of $S_k$. The first region is
    isometric to $P_{k-1}$ and can be tiled by the parallelograms $v_i
    \times v_j$, $i < j < k$, by the induction hypothesis. The second
    region is the sum $\partial^+ P_{k-1} + S_k$, and is tiled by the
    parallelograms $v_i \times v_k$, for all $i < k$, since
    $\partial^+ P_{k-1}$ is a concatenation of the vectors $v_i$, $i <
    k$. \qedhere 

  \end{proof}
  
  Since $P_\alpha(q)$ is $P_k$ for the vectors $\set{m_e v_e \st e
  \text{ appears in }\alpha}$, ordered appropriately, we immediately
  deduce the following from Lemma \ref{tiling}\ref{t2}. 

  \begin{lemma}\label{area-sum}

   For any $\alpha \in \scc$ and $q \in \QD^1(S)$, we have
   \[
     \area(P_\alpha(q)) = \sum_{e, e'} m_em_{e'} \abs{v_e \wedge v_{e'}}, 
   \]
   where the sum is taken over all unordered pairs of distinct saddle
   connections $e, e'$ appearing in $\alpha^q$. \qed

  \end{lemma}

  \subsection{Rigidity}
   
  A key tool that we use is the following result of Duchin, Leininger, and
  Rafi. Our formulation of the statement is slightly different from theirs, 
  but their proof still works without modification. 
  
  \begin{proposition}[\cite{DLR10}, Lemma 22]\label{rigidity}

    Let $q, q' \in \QD^1(S)$ be half-translation surfaces. If\/
    $\hcyl(q) \subset \hcyl(q')$ then $\SL \cdot q = \SL \cdot
    q'$. \qed 

  \end{proposition}

  Therefore, if $\SL \cdot q \not= \SL \cdot q'$ then both
  $\hcyl(q) - \hcyl(q')$ and $\hcyl(q') - \hcyl(q)$ are non-empty. 

  \begin{remark}

    One consequence of Proposition \ref{rigidity} is that when $3g-3+p >
    1$, every $q \in \QD^1(S)$ has a crooked curve. This is true
    because $\QD^1(S)$ always has real dimension greater than $3$, and
    hence cannot be a single $\SL$--orbit. 

  \end{remark}
      
  \begin{remark}
   
    Theorem 1 of \cite{DLR10} states that the marked length spectrum of
    simple closed curves determines the half-translation surface $q$.
    Associated to $q$ is the \emph{marked polygonal area spectrum}, which
    is the $\SL$--invariant function $\scc \to \R$ given by $\alpha \mapsto
    \area(P_\alpha(q))$. We observe that, by Proposition
    \ref{polygon}\ref{tw4} and Proposition \ref{rigidity}, the marked
    polygonal area spectrum of $q$ determines its $\SL$--orbit. 

  \end{remark}
      
  \subsection{$\SL$--orbit closures}
  
  The space $\QD^1(S)$ of unit-area half translation structures on $S$ is
  naturally partitioned into strata $\Qk$, where $\kappa$ is a partition of
  $4g - 4$ specifying the orders of the singularities. The mapping class
  group $\mcg(S)$ acts on $\QD^1(S)$ and each stratum $\Qk$ by change of
  marking. The $\SL$--action on $\QD^1(S)$ preserves each stratum, and also
  descends naturally to the moduli space of half-translation surfaces
  $\QMod(S) = \QD^1(S) / \mcg(S)$, as well as each unmarked stratum $\QkMod
  = \Qk / \mcg(S)$. Let $\pi \from \QD^1(S) \to \QMod(S)$ be the natural
  projection, which is an orbifold covering map. Given $q \in \Qk$, let
  $\aim_q$ be the closure of $\pi(\SL \cdot q)$ in $\QkMod$. We call
  $\aim_q$ the \emph{orbit closure} associated to $q$. 
  
  The structure of $\aim_q$ has been elucidated in the work of Eskin,
  Mirzakhani, and Mohammadi, as follows: 

  \begin{theorem}[\cite{EMM15}, Theorem 2.1]\label{EMM}

    For any $q \in \Qk$, the orbit closure $\aim_q$ is an affine
    invariant submanifold of\/ $\QkMod$. \qed

  \end{theorem}
  
  The statement given in \cite{EMM15} actually refers to abelian
  differentials rather than quadratic differentials. However, the result
  applies equally well to the setting of quadratic differentials, by
  considering an appropriate two-fold branched covering of the surface. 

  The precise definition of ``affine invariant submanifold'' is rather
  involved and we shall not repeat it here in its entirety. It
  suffices to note that it includes the following: 

  \begin{itemize}

    \item $\aim_q$ is $\SL$--invariant, 

    \item $\aim_q$ is the image of a properly immersed orbifold
      $f \from N \to \QkMod$. 

  \end{itemize}

  Since $\pi \from \Qk \to \QkMod$ is an orbifold covering, the
  preimage $\pi^{-1} (\aim_q)$ is also the image of a properly
  immersed orbifold (in fact, manifold) of the same dimension as
  $N$. 

  The main conclusion we need to draw from Theorem \ref{EMM} is that
  $\pi^{-1}(\aim_q)$ is locally path connected. Let $\aimup_q$ be the
  connected component of $\pi^{-1}(\aim_q)$ containing $q$. Then local
  path connectedness of $\pi^{-1}(\aim_q)$ implies that $\aimup_q$ is
  open and closed in $\pi^{-1}(\aim_q)$. 

  Now let $\Gamma_q \leq \mcg(S)$ be the (setwise) stabilizer of
  $\aimup_q$, and define 
  \[
    \orb_q = \Gamma_q \cdot (\SL \cdot q).
  \] 

  \begin{lemma}\label{EMMlemma}

    $\aimup_q$ is the closure of\/ $\orb_q$ in $\Qk$. 

  \end{lemma}

  \begin{proof} 

    Certainly, $\aimup_q$ contains $\orb_q$ and is closed in $\Qk$, so
    it contains the closure of $\orb_q$. Also, $\pi^{-1}(\aim_q)$ is
    the closure of $\mcg(S) \cdot (\SL \cdot q)$. Note that if $g \in
    \mcg(S) - \Gamma_q$ then $g \cdot (\SL \cdot q)$ is contained in
    the component $g \aimup_q$ of $\pi^{-1}(\aim_q)$, which is
    disjoint from $\aimup_q$. Now if $q' \in \aimup_q$ is a limit of a
    sequence of points $q_i \in g_i \cdot (\SL \cdot q)$, the open set
    $\aimup_q$ must contain almost all $q_i$, and therefore $g_i \in
    \Gamma_q$ for almost all $i$, and $q'$ is in the closure of
    $\orb_q$. \qedhere 
    
  \end{proof}
  
  We will also make use of the following finiteness result.
  
  \begin{theorem}[\cite{EMM15}, Theorem 2.2]\label{EMMfinite}
   
   Any closed $\SL$--invariant subset of\/ $\Qk$ is a finite union
   of\/ $\SL$--orbit closures. \qed 
   
  \end{theorem}

  \subsection{Veech surfaces}
  
  Recall that $q\in \QD(S)$ is a \emph{Veech surface} if its group of
  affine automorphisms is a lattice in $\SL$. We now state several
  characterizations of Veech surfaces due to Smillie and Weiss
  \cite{SW10}, which builds on work of Vorobets \cite{Vor96}. By a
  \emph{triangle} on $q$, we mean a Euclidean triangle on $q$ with
  isometrically embedded interior, whose sides are saddle connections
  on $q$. Let $\hol(q)$ be the set of holonomy vectors arising from
  saddle connections on $q$.
  
  \begin{theorem}[\cite{SW10}]\label{nst}

   For any $q \in \Qk$, the following are equivalent.

    \begin{enumerate}[label=\textup{(\roman*)}]
      \item $q$ is a Veech surface, \label{sw1} 
      \item $q$ has no small triangles: there is a lower bound $\epsilon >
        0$ on the areas of all triangles on $q$, 
      \item $q$ has no small virtual triangles: there exists $\epsilon > 0$
        such that $\abs{{u} \wedge {v}} > \epsilon$ for all pairs
        of non-parallel holonomy vectors $ {u}, {v} \in
        \hol(q)$, \label{sw3} 
      \item the virtual triangle area spectrum $\VT(q) = \set{\abs{{u}
            \wedge {v}} \st {u}, {v} \in 
        \hol(q)}$ is discrete, \label{sw3a}
      \item $\pi(\SL \cdot q)$ is closed in $\QkMod$. \label{sw4} 
        \qed
   \end{enumerate}

  \end{theorem}
  
  Note that condition \ref{sw4} is the same as saying that 
  $\aim_q = \pi(\SL \cdot q)$. Applying Lemma \ref{EMMlemma}, we deduce
  the following.
  
  \begin{corollary}\label{veechcor}
  
    A half-translation surface $q$ is a Veech surface if and only
    if\/ $\aimup_q = \orb_q$. \qed
    
  \end{corollary}

\section{Continuity results}

  \label{sec:continuity}

  \subsection{Continuity of the auxiliary polygon}

  Here we show that the polygon $P_\alpha(q)$ is continuous in $q$,
  with respect to the Hausdorff topology in the plane. First we recall
  some basic notions from convex geometry. See, for instance, Sections
  1.7 and 1.8 of \cite{Sch14}. 

  For any non-empty compact convex subset $K \subset \R^2$, the
  \emph{support function} $h_K \from S^1 \to \R$ is defined by 
  \[ 
    h_K(u) = \sup \set{\langle x, u\rangle \st x \in K}.
  \]
  Here $\langle \, \param \, , \param \, \rangle$ is the usual
  inner product on $\R^2$. The \emph{width function} $w_K \from S^1
  \to \R$ is defined by 
  \[
    w_K(u) = h_K(u) + h_K(-u).
  \]
  Note that $w_K$ is even, and descends to a function on $\RP$ which we
  also denote by $w_K$. Now let $d_H$ denote Hausdorff distance. We have
  the following standard fact: 

  \begin{lemma}[\cite{Sch14},Lemma~1.8.14]\label{hausdorff}

    Suppose $K$ and $L$ are non-empty compact convex subsets of\/ 
    $\R^2$. Then  
    \begin{equation}
      d_H(K,L) = \sup_{u \in S^1} \abs{h_K(u) - h_L(u)}. \tag*{\qed}
    \end{equation}

  \end{lemma}

  Next consider \emph{centrally symmetric} convex sets: these are
  convex sets $K$ such that $K = -K$. Note that the auxiliary polygons
  $P_\alpha(q)$ are both convex and centrally symmetric. If $K$ is
  centrally symmetric then $h_K(u) = h_K(-u)$ for all $u$, and
  therefore 
  \begin{equation}\label{heightwidth}
    w_K = 2h_K.
  \end{equation}

  Now let $\norm{f}$ denote the sup norm for functions $f \from \RP
  \to \R$, and let $\cont(\RP,\R)$ be the space of continuous
  functions, with the sup metric. Let $\convexspace$ be the space of
  non-empty centrally symmetric compact convex sets in $\R^2$, with
  the Hausdorff metric. The next lemma follows directly from Lemma
  \ref{hausdorff} and equation \eqref{heightwidth}. 

  \begin{lemma}\label{hausdorffwidth}

    If $K, L \in \convexspace$ then 
    \begin{equation} 
      2d_H(K,L) = \sup_{\theta \in \RP} \abs{w_K(\theta) - w_L(\theta)} =
      \norm{w_K - w_L}. \tag*{\qed}
    \end{equation}

  \end{lemma}

  \begin{corollary}\label{embedding}

    The map $W \from \convexspace \to \cont(\RP,\R)$ given by $K \mapsto
    \frac{1}{2}w_K$ is an isometric embedding. \qed

  \end{corollary}

  Let us now return our attention to the auxiliary polygons. 

  \begin{theorem}\label{polygoncontinuity}
    
    The map $\QD^1(S) \times \scc \to \convexspace$ defined by $(q,
    \alpha) \mapsto P_\alpha(q)$ is continuous in the first factor. 

  \end{theorem}

  \begin{proof}

  Let $\alpha \in \scc$ be fixed. Applying Proposition \ref{polygon},
  we see that 
  \[
    w_{P_\alpha(q)}(\theta) = \width(e^{-i\theta}\cdot P_\alpha(q)) =
    l_\alpha^H(e^{-i\theta}\cdot q) = i\big(\nu_q^{\frac{\pi}{2}+\theta},
    \alpha\big)
  \] 
  for all $q \in \QD^1(S)$ and $\theta \in \RP$. The map $\QD^1(S) \times
  \RP \to \MF(S)$ given by $(q, \theta) \mapsto
  \nu_q^{\frac{\pi}{2}+\theta}$ is continuous. By continuity of
  intersection number on $\MF(S) \times \MF(S)$ and compactness of
  $\RP$, the map $q \mapsto w_{P_\alpha(q)}$ defines a continuous
  function from $\QD^1(S)$ to $\cont(\RP,\R)$. Moreover, its image is
  contained in $W(\convexspace)$, and composing this map with
  $\frac{1}{2}W^{-1}$ yields the function $q \mapsto
  P_\alpha(q)$. This map is continuous by Corollary
  \ref{embedding}. \qedhere  
  
  \end{proof}

  Finally, applying continuity of $\area \from \convexspace \to \R$
  \cite[Theorem~1.8.20]{Sch14} yields the desired result.

  \begin{corollary}\label{areacontinuity}
    
    The function $\area \from \QD^1(S) \times \scc \rightarrow \R_{\geq 0}$
    defined by $\area(q, \alpha) = \area(P_\alpha(q))$ is continuous and
    $\SL$--invariant in the first factor. \qed

  \end{corollary}
  
  \subsection{Continuity of \SL--infimal length}
  
  In this section, we apply continuity of the auxiliary polygon to deduce
  continuity of\/ $\SL$--infimal length.

  Given $K \in \convexspace$, let $r^-(K) = \inf_{u \in S^1} h_K(u)$ and
  $r^+(K) = \sup_{u \in S^1} h_K(u)$. One can show that these two numbers
  coincide with the minimum and maximum distances to the origin of
  points on $\partial K$. Define the \emph{eccentricity} of $K$ 
  to be $\ecc(K) = \frac{r^+(K)}{r^-(K)} \geq 1$. Note that $r^-(K) =
  0$ if and only if $K$ is degenerate. 

  Recall that any matrix $A \in \SL$ has a \emph{singular value
  decomposition} 
  \[
    A \ = \ e^{i\theta_1} \begin{pmatrix}\lambda & 0 \\
    0 & \lambda^{-1} \end{pmatrix} e^{i\theta_2},
  \]
  for some \emph{stretch factor} $\lambda \geq 1$ and $\theta_1,
  \theta_2 \in \R$. Moreover, $\lambda = \sv(A)$ is unique, and
  $\sv(AB) \leq \sv(A)\sv(B)$ and $\sv(A^{-1}) = \sv(A)$ for all $A,B \in
  \SL$. Also note that
  \begin{equation}\label{radiuslambda}
    \sv(A)r^-(K)  \leq    r^+(A \cdot K)  \leq  \sv(A)r^+(K) \quad
    \textrm{and} \quad 
    \frac{r^-(K)}{\sv(A)}  \leq  r^-(A \cdot K)  \leq  \frac{r^+(K)}{\sv(A)}
  \end{equation}
  for all $K \in \convexspace$.
  Since $A$ acts as a $\sv(A)$--Lipschitz map from $\R^2$ to itself, we deduce
  for all $K,L \in \convexspace$ that
  \begin{equation}\label{hdstretch}
    d_H(A \cdot K, A \cdot L) \ \leq \ \sv(A)~ d_H(K,L).
  \end{equation}

  \begin{lemma}\label{ecchaus}

    Fix $r_0 > 0$, and suppose $K,L \in \convexspace$ satisfy $r^-(K),
    r^-(L) > r_0$. Then 
    \[
      \left(1 + \frac{d_H(K,L)}{r_0}\right)^{-2} < \
      \frac{\ecc(K)}{\ecc(L)} \ < \ \left(1 +
        \frac{d_H(K,L)}{r_0}\right)^2. 
    \]

  \end{lemma}
   
  \begin{proof}

    Let $D = d_H(K,L)$. Applying Lemma \ref{hausdorff} gives
    \[ 
      \frac{\ecc(K)}{\ecc(L)} \ = \
      \frac{r^+(K)}{r^+(L)}\frac{r^-(L)}{r^-(K)} \ \leq \
      \left(\frac{r^+(L) + D}{r^+(L)}\right)\left(\frac{r^-(K) + 
      D}{r^-(K)}\right) \ < \ \left(1 + \frac{D}{r_0}\right)^2.
    \] 
    The other bound can be deduced similarly. \qedhere
   
  \end{proof}
  
  Define the \emph{eccentricity} of a curve $\alpha \in \scc$ on $q \in
  \QD^1(S)$ to be $\ecc_\alpha(q) = \ecc(P_\alpha(q))$. Note that
  $\ecc_\alpha(q) = \infty$ if and only if $\alpha \in \hcyl(q)$. By
  Theorem \ref{polygoncontinuity}, $\ecc_\alpha$ is continuous on the set
  of quadratic differentials on which $\alpha$ is crooked.
   
  \begin{lemma}\label{lambdagrowth}
  
    Let $\alpha$ be a crooked curve on $q \in \QD^1(S)$, and suppose the
    \SL--infimal length of $\alpha$ is attained at $m \in \SL\cdot q$.
    There is a constant $c \geq 1$ (independent of\/ $S$, $q$, and
    $\alpha$) such that
    \[
      \frac{\sv(A)^2}{c} \ \leq \ \ecc_\alpha(A \cdot m) \ \leq \
      c\sv(A)^2
    \]
    for all $A \in \SL$. In particular, $\ecc_\alpha(m) \leq c$.
  
  \end{lemma}
  
  \begin{proof}
  
    By Lemma 6.4 and Corollary 7.2 of \cite{TW15}, there exists $A_0 \in
    \SL$ such that $\ecc_\alpha(A_0 \cdot m) \leq 2$ and $\sv(A_0) < c'$
    for a constant $c'$ independent of $S$, $q$, and
    $\alpha$. Taking $c \geq 2(c')^2$, the result 
    follows by applying the inequalities in
    (\ref{radiuslambda}), with $K = P_\alpha (A_0 \cdot m)$. \qedhere 
  
  \end{proof}
  
  \begin{lemma}\label{fellow}
   
    Let $\alpha$ be a crooked curve on $q \in \QD^1(S)$. Then for all
    $\epsilon > 0$ and $\lambda_0 \geq 1$, there exists an open
    neighborhood $U \subseteq \QD^1(S)$ of $q$ such that for all $q'
    \in U$ and $A \in \SL$ satisfying $\sv(A) \leq \lambda_0$, we have:
    \begin{enumerate}[label=\textup{(\roman*)}]
      \item $(1+\epsilon)^{-2} < \frac{\ecc_\alpha(A \cdot
        q')}{\ecc_\alpha(A \cdot q)} < (1 + \epsilon)^2$, \label{f2}
      \item $|l_\alpha(A \cdot q) - l_\alpha(A \cdot q')| < \pi
        \epsilon$. \label{f3} 
    \end{enumerate}
   
  \end{lemma}
   
  \begin{proof}
    
    Set $r_0 = \frac{r^-(P_\alpha(q))}{2\lambda_0}$, and choose $\delta <
    \min \set{\frac{r^-(P_\alpha(q))}{2},
      \frac{r_0\epsilon}{\lambda_0}, \frac{\epsilon}{\lambda_0}}$.
    
    By Theorem \ref{polygoncontinuity}, there is an open neighborhood $U$
    of $q$ with $d_H(P_\alpha(q), P_\alpha(q')) < \delta$ for all $q' \in
    U$. Choose any $q' \in U$ and $A \in \SL$ satisfying $\sv(A) \le
    \lambda_0$. Then 
    \[
      d_H(P_\alpha(A \cdot q), P_\alpha(A \cdot q')) < \min
      \set{ r_0 \epsilon, \epsilon}
    \] 
    by \eqref{hdstretch}. Now by \eqref{radiuslambda}, we deduce
    \[
      r^-(P_\alpha(A \cdot q')) \ \geq \
      \frac{r^-(P_\alpha(q'))}{\sv(A)} \ 
      > \ \frac{r^-(P_\alpha(q))-\delta}{\lambda_0} \ > \
      \frac{r^-(P_\alpha(q))}{2\lambda_0} \ = \ r_0. 
    \] 
    Applying Lemma \ref{ecchaus} yields the first claim. For the second
    claim, we use the well-known fact that the perimeter of any convex
    region $K \subset \R^2$ is $\int_0^\pi w_K(u) du$, where $w_K(u)$ is
    the width of $K$ at $u \in S^1$. Set $w_q(u)$ to be the width of
    $P_\alpha(q)$ at $u$. Then using Proposition
    \ref{polygon}\ref{tw1} and Lemma \ref{hausdorffwidth} we have 
    \begin{align*} 
      |l_\alpha(A\cdot q) - l_\alpha(A \cdot q')| \ 
      & = \ \frac{1}{2} \left| \int_0^\pi w_{A\cdot q}(u) - w_{A \cdot
        q'}(u) du \right| \\
      & \le \ \frac{1}{2} \int_0^\pi \left| w_{A\cdot q}(u) - w_{A \cdot
      q'}(u) \right| du \\
      & \le \ \pi d_H(P_\alpha(A\cdot q), P_\alpha(A \cdot q')) \ < \
        \pi \epsilon. \qedhere
    \end{align*}
   
  \end{proof}

  \begin{proposition}\label{infimalcontinuity}
    
    For any curve $\alpha \in \scc$, the function $\lsl_\alpha \from
    \QD^1(S) \rightarrow \R_{\geq 0}$ is continuous and $\SL$--invariant.

  \end{proposition}
  
  \begin{proof}

    Fix $\alpha \in \scc$ and $q \in \QD^1(S)$. We now prove that
    $\lsl_\alpha$ is continuous at $q$. In the case where $\alpha \in
    \hcyl(q)$, we have $\lsl_\alpha(q) = 0 = \area(q,\alpha)$. Continuity
    of $\lsl_\alpha$ at $q$ then follows from Proposition
    \ref{polygon}\ref{tw3} and Corollary \ref{areacontinuity}.
    
    Assume $\alpha$ is crooked on $q$. Given any $\epsilon \in (0,1)$,
    we shall show that $|\lsl_\alpha(q) - \lsl_\alpha(q')| < \pi
    \epsilon$ for all $q'$ in a sufficiently small neighborhood of
    $q$. Suppose $\lsl_\alpha(q)$ is attained at $A_1 \cdot q$, for
    some $A_1 \in \SL$. Choose $\lambda_0 \geq 2c\sv(A_1)$,
    where $c$ is the constant from Lemma \ref{lambdagrowth}. Let $U$ be an
    open neighborhood of $q$ which satisfies the conclusion of Lemma
    \ref{fellow}. For $q' \in U$, let $A_2 \in \SL$ be such that
    $\lsl_\alpha(q') = l_\alpha(A_2 \cdot q')$. By Lemmas
    \ref{lambdagrowth} and \ref{fellow}\ref{f2}, we have 
    \[ 
      \ecc_\alpha(A_1 \cdot q') \ < \ (1+\epsilon)^2 \ecc_\alpha(A_1
      \cdot q) \ < \ (1+\epsilon)^2 c \ < \ 4c. 
    \] 
    Applying Lemma \ref{lambdagrowth} with $m = A_2 \cdot q'$, we
    obtain $\lambda(A_1 A_2^{-1})^2 \leq c \cdot\ecc_\alpha (A_1 \cdot
    q')$, and therefore $\sv(A_1A_2^{-1}) \leq 2c$. Thus, 
    \[ 
      \sv(A_2) \ \leq \ \sv(A_2A_1^{-1})\sv(A_1) \ = \
      \sv(A_1A_2^{-1})\sv(A_1) \ \leq \ 2c\sv(A_1) \ \leq \
      \lambda_0.
    \]  
    Finally, applying Lemma \ref{fellow}\ref{f3} yields
    \begin{multline*}
    \lsl_\alpha(q) - \pi \epsilon \ \leq \ l_\alpha(A_2 \cdot q) - \pi
    \epsilon \ < \ l_\alpha(A_2 \cdot q') \ = \ \lsl_\alpha(q') \\
    \leq \ l_\alpha(A_1 \cdot q') \ < \ l_\alpha(A_1 \cdot q) + \pi
    \epsilon \ = \ \lsl_\alpha(q) + \pi \epsilon 
    \end{multline*}
    as required. \qedhere

  \end{proof}
  
\section{The polygonal area spectrum}
  
  We are now ready to prove Theorem \ref{nscc}. The first step is
  Theorem \ref{discreteness}, which says that $\polyspec(q)$ is
  discrete if $q$ is a Veech surface. 
  
  \proof[Proof of Theorem \ref{discreteness}] 
  
  By Theorem \ref{nst}\ref{sw3a}, the virtual triangle area spectrum
  $\VT(q)$ is discrete. For each simple closed curve $\alpha$, 
  $\area(P_{\alpha}(q))$ is a positive integer combination of numbers
  from the set $\VT(q)$, by Lemma \ref{area-sum}. The result follows. 
  \qedhere
  
  \endproof

  Using Proposition \ref{polygon}\ref{tw3}, we obtain: 

  \begin{corollary}

   If $q \in \QD^1(S)$ is a Veech surface then
   \begin{equation}\inf\set{\lsl_\alpha(q) \st \alpha \text{ is
         crooked on } q} > 0.\tag*{\qed}\end{equation} 

  \end{corollary}

  For the converse, it is worth remarking that the existence of short 
  crooked curves is not an immediately obvious consequence of having small
  virtual triangles. From a given collection of saddle connections on
  $q$ for which $\abs{{u} \wedge {v}}$ can be taken to be arbitrarily
  small, it appears difficult to construct a sequence of saddle
  connections to satisfy the following:
  \begin{itemize}

   \item no two saddle connections intersect (in their interiors),

   \item consecutive saddle connections meet with an angle of at least
     $\pi$ on both sides, 

   \item their concatenation is homotopic to an essential simple
     closed curve. 

  \end{itemize}
  (Small triangles on $q$ are not particularly useful since their sides
  must meet at an angle of less than $\pi$.) In our proof below, the
  auxiliary polygon plays a key role in bypassing this difficulty.

  For $q \in \Qk$, recall that $\orb_q = \Gamma_q \cdot (\SL \cdot q)$
  is a dense subset of $\aimup_q$ in $\Qk$, where $\Gamma_q \leq \mcg(S)$
  is the stabilizer of $\aimup_q$.
  
  \begin{proposition}\label{notveech-dense}

    Suppose $q\in \QD^1(S)$ is not a Veech surface. Then there is a
    number $a > 0$ such that the polygonal area spectrum $\polyspec(q)$
    contains a dense subset of\/ $[0,a]$.

  \end{proposition}
  
  \proof

    Applying Corollary \ref{veechcor}, we have $\aimup_q \neq \orb_q$
    and so we may choose $q' \in \aimup_q - \orb_q$. Then $\SL \cdot q \neq
    \SL \cdot q'$, and so by Proposition \ref{rigidity}, there exists
    a curve $\alpha \in \hcyl(q) - \hcyl(q')$. By Proposition
    \ref{polygon}\ref{tw4}, we have $\area(q, \alpha) = 0$ and 
    $\area(q', \alpha) = a > 0$. Since $\aimup_q$ is connected and
    $\area(\, \param \, , \alpha)$ is continuous, by Corollary
    \ref{areacontinuity}, we deduce that $[0,a] \subseteq
    \area(\aimup_q, \alpha)$. It follows that $\area(\orb_q, \alpha)$,
    and hence $\area(\orb_q, \scc)$, contains a dense subset of
    $[0,a]$. Finally, the polygonal area spectrum $\polyspec(q) = 
    \area(q, \scc)$ is invariant under $\SL$--deformations and changes 
    of markings, and therefore $\area(q, \scc) = \area(\orb_q,
    \scc)$. \qedhere 

  \endproof
  
  Note that this proof uses only the fact that $\area(\, \param \, , \alpha)$
  is a continuous \SL--invariant function to $\R_{\geq 0}$ which takes
  the value 0 precisely when $\alpha \in \hcyl(q)$. Since the same properties
  hold for $\lsl_\alpha$, we may argue as above to deduce:
  
   \begin{proposition}

   If $q \in \QD^1(S)$ is not a Veech surface then there exists
   $a > 0$ such that the \SL--infimal length spectrum
   $\set{\lsl_\alpha(q) \st \alpha \in \scc}$
   contains a dense subset of \/ $[0,a]$. \qed
  
   \end{proposition}   
  
  We conclude this section with a proof of Theorem \ref{finiteness},
  which is restated below. For $a > 0$, recall that
  \begin{eqnarray*}
    \PA(a) &=& \set{q \in \QD^1(S) \st \area(\poly_\alpha(q)) \geq a \text{
    for every crooked curve } \alpha \text{ on } q}, \\
    \IL(a) &=& \set{q \in \QD^1(S) \st \lsl_\alpha(q) \geq a \text{
    for every crooked curve } \alpha \text{ on } q}.
  \end{eqnarray*}
  
  \begin{theorem}
   
   For any $a>0$, the sets $\PA(a)$ and $\IL(a)$ both contain only finitely
   many affine equivalence classes of half-translation surfaces.
   
  \end{theorem}

  \proof
  
    Applying Corollary \ref{areacontinuity}, we deduce that $\PA(a)$ is a
    closed and $\SL$--invariant subset of $\QD^1(S)$. Moreover, $\PA(a)$
    is invariant under the action of $\mcg(S)$, and so $\PA(a)$ descends
    to a closed $\SL$--invariant subset $C$ in $\QMod(S)$. It follows that
    $C \cap \Qk$ is closed and $\SL$--invariant in each stratum $\Qk$
    under the subspace topology. By Theorem \ref{EMMfinite}, $C \cap \Qk$
    is a finite union of $\SL$--orbit closures. Since elements of $\PA(a)$
    are necessarily Veech surfaces by Theorem \ref{nscc}, $C \cap \Qk$
    must be a finite union of $\SL$--orbits. Finally, there are only finitely
    many strata for a given genus, and so the desired result follows.
    
    The proof for $\IL(a)$ proceeds identically using Proposition
    \ref{infimalcontinuity} in place of Corollary
    \ref{areacontinuity}.\qedhere 
    
  \endproof
  
\section{The virtual triangle area spectrum} 

  \label{sec:triangles}
  
  We are almost ready to prove Theorem \ref{nsvt}. In the introduction we
  defined the virtual triangle area spectrum $\VT(q)$. Let $\VT_0(q)
  \subset \VT(q)$ be the subset consisting of the numbers $\abs{u \wedge
  v}$ such that $u$ and $v$ are the holonomy vectors of a pair of saddle
  connections with a common endpoint (that is, a \emph{based} virtual
  triangle). Note that saddle connections forming a based virtual triangle
  need not form a triangle, since they may have angle $\pi$ or more on both
  sides. We also define $\VT_0(X) = \bigcup_{q \in X}\VT_0(q)$ for any set
  $X\subset \QD^1(S)$.
  
  We know from the implication $\ref{sw1} \Rightarrow \ref{sw3}$ of Theorem
  \ref{nst} that if $q$ is a Veech surface then $\VT(q)$ has a gap above
  zero. The remainder of Theorem \ref{nsvt} follows from the next
  proposition: 

  \begin{proposition}\label{densetriangles}

    Suppose $q \in \QD^1(S)$ is not a Veech surface. Then there is a
    number $a > 0$ such that $\VT_0(q)$ contains a dense subset of\/
    $[0,a]$. 

  \end{proposition}
  
  \begin{proof} 
  
    As in the proof of Proposition \ref{notveech-dense}, there exist a
    half-translation surface $q' \in \aimup_q - \orb_q$ and a curve
    $\alpha \in \hcyl(q) - \hcyl(q')$, and we know that $\area(q,
    \alpha) = 0$ and $\area(q',\alpha) > 0$. Let $q_t$ be a path in
    $\aimup_q$ from $q_0 = q$ to $q_1 = q'$. Consider the set 
    \[
      \set{t \in [0,1] \st \area(q_t, \alpha) > 0},
    \]
    which is an open neighborhood of $1$ in $[0,1]$ that does not
    contain $0$. It has a connected component $(t_0, 1]$. Replacing 
    the path $q_t$ by its restriction to $[t_0, 1]$ and
    reparametrizing over $[0,1]$, we have $\area(q_0, \alpha) = 0$ and
    $\area(q_t, \alpha) > 0$ for all $t \in (0,1]$. 

    Now consider the geodesic representative $\alpha^{q_0}$ and
    express it as a concatenation of saddle connections $e_1 \dotsm
    e_k$. Let $\alpha_i$ be the topological arc represented by
    $e_i$; it is an isotopy class rel endpoints, where the interior of
    the arc is required to avoid the singularities. Consider the set 
    \[
      \set{t \in [0,1] \st \text{each } \alpha_i \text{ is represented
          by a saddle connection in } q_t}.
    \]
    This set is an open neighborhood of $0$ in $[0,1]$. Again,
    replacing the path $q_t$ by its restriction to an interval $[0,
    \epsilon]$ and reparametrizing over $[0,1]$, we may assume that
    the arcs $\alpha_1, \dotsc, \alpha_k$ are represented by saddle
    connections for all $t \in [0,1]$. Let $e_i(t)$ denote the saddle
    connection in $q_t$ representing $\alpha_i$. Let $v_i(t)$ be the
    holonomy vector of $e_i(t)$. 

    Define functions $\phi_i(t) = \abs{v_i(t) \wedge v_{i+1}(t)}$ for
    each $i$ (with indices taken mod $k$).
    These are continuous because holonomy vectors vary continuously
    where defined. Since all $e_i$ are parallel on $q_0$,
    we have $\phi_i(0) = 0$ for all $i$.
    We claim that $\phi_i(t_1) > 0$ for some $t_1>0$ and some $i$.
    If not, then for all $t \in [0,1]$ the saddle
    connections $e_i(t)$ are all parallel. The angles
    between consecutive saddle connections (on either side) must
    remain constant, since they are constrained to lie in the discrete
    set $\pi\Z$, and therefore the concatenation $e_1(t) \dotsm
    e_k(t)$ remains a geodesic representative for $\alpha$ on $q_t$. 
    But this contradicts the fact that $\area(q_t, \alpha)
    > 0$ for all $t \in (0,1]$. 

    Restricting $q_t$ to $[0,t_1]$ and reparametrizing over $[0,1]$ one
    last time, we have a path $q_t$ in $\aimup_q$ and a pair of saddle
    connections $e_i(t)$, $e_{i+1}(t)$ which persist on $q_t$ throughout
    the path, such that $\phi_i(0) = \abs{v_i(0) \wedge v_{i+1}(0)} = 0$
    and $\phi_i(1) = \abs{v_i(1) \wedge v_{i+1}(1)} = a > 0$.  The function
    $\phi_i$ is defined and continuous on the open set $U \subset \aimup_q$
    where $e_i$ and $e_{i+1}$ persist. This set contains the path $q_t$ and
    hence $\phi_i(U)$ contains $[0,a]$. Since $U \cap \orb_q$ is dense in
    $U$, it follows that $\phi_i(U \cap \orb_q)$, and hence
    $\VT_0(\orb_q)$, contains a dense subset of $[0,a]$. Finally, the based
    virtual triangle area spectrum is invariant under change of marking and
    $\SL$, and so $\VT_0(q) = \VT_0(\orb_q)$. \qedhere

  \end{proof}

  \small

\providecommand{\bysame}{\leavevmode\hbox to3em{\hrulefill}\thinspace}
\providecommand{\MR}{\relax\ifhmode\unskip\space\fi MR }
\providecommand{\MRhref}[2]{%
  \href{http://www.ams.org/mathscinet-getitem?mr=#1}{#2}
}
\providecommand{\href}[2]{#2}

\end{document}